\documentclass[a4paper,10pt]{amsart}
\usepackage{amssymb,amsfonts,amsmath}
\usepackage{latexsym}
\usepackage{amssymb}
\usepackage[all]{xy}
\newtheorem{theorem}{Theorem}[section]
\newtheorem{lemma}[theorem]{Lemma}
\newtheorem{proposition}[theorem]{Proposition}
\newtheorem{corollary}[theorem]{Corollary}
\theoremstyle{definition}
\newtheorem{definition}[theorem]{Definition}

\theoremstyle{remark}
\newtheorem{remark}[theorem]{Remark}

\newcommand{\Mod}{\text{\rm mod-}}

\newcommand{\Proj}{\operatorname{Proj}}
\newcommand{\Hom}{\operatorname{Hom}}

\newcommand{\Ext}{\operatorname{Ext}}
\newcommand{\End}{\operatorname{End}}
\newcommand{\Ker}{\operatorname{Ker}}

\newcommand{\Ima}{\operatorname{Im}}

\title{Combinatorial aspects of extensions of Kronecker modules}

\author{Csaba Sz\'ant\'o}

\begin{document}
\maketitle \pagestyle{myheadings} \markboth{\sc Csaba
Sz\'ant\'o}{\sc Combinatorial aspects of extensions of Kronecker modules}

{\bf Abstract.} Let $kK$ be the path algebra of the Kronecker quiver and consider the category $\Mod
kK$ of finite dimensional right modules over $kK$ (called Kronecker modules). We prove that extensions of Kronecker modules are field independent up to Segre classes, so they can be described purely combinatorially. We use in the proof explicit descriptions of particular extensions and a variant of the well known Green formula for Ringel-Hall numbers, valid over arbitrary fields.  We end the paper with some results on extensions of preinjective Kronecker modules, involving the dominance ordering from partition combinatorics and its various generalizations.


\medskip


\section{\bf Introduction}
Let $K$ be the Kronecker quiver and $k$ a field. We will consider the path algebra $kK$ of $K$ over $k$
(called Kronecker algebra) and the category $\Mod kK$ of finite
dimensional right modules over $kK$ (called Kronecker modules). The isomorphism class of the module $M$ will be denoted by $[M]$.

For $d\in\mathbb N^2$ let $M_d$ be the set of isomorphism classes of
Kronecker modules of dimension $d$. Following Reineke in
\cite{reineke} for subsets $\mathcal{A}\subset M_d$,
$\mathcal{B}\subset M_e$ we define
$$\mathcal{A}*\mathcal{B}=\{[X]\in M_{d+e}|\text{there exists an
exact sequence } 0\to N\to X\to M\to 0\text{ for some }
[M]\in\mathcal{A},[N]\in\mathcal{B}\}.$$ So the product
$\mathcal{A}*\mathcal{B}$ is the set of isoclasses of all extensions
of modules $M$ with $[M]\in\mathcal{A}$ by modules $N$ with
$[N]\in\mathcal{B}$. This is in fact Reineke's extension monoid
product using isomorphism classes of modules instead of modules. It
is important to know (see \cite{reineke}) that the product above is
associative, i.e. for $\mathcal{A}\subset M_d$, $\mathcal{B}\subset
M_e$, $\mathcal{C}\subset M_f$, we have
$(\mathcal{A}*\mathcal{B})*\mathcal{C}=\mathcal{A}*(\mathcal{B}*\mathcal{C})$. We also have $\{[0]\}*\mathcal{A}=\mathcal{A}*\{[0]\}=\mathcal{A}$ and the product $*$ is distributive over the union of sets.

Recall that in case when $k$ is finite, the rational Ringel-Hall algebra $\mathcal{H}(\Lambda,\mathbb Q)$
associated to the algebra $kK$, is the free $\mathbb Q$-module
having as basis the isomorphism classes of Kronecker modules together
with a multiplication defined by $[N_1][N_2]=\sum_{[M]}F^M_{N_1N_2}[M]$, where the structure constants $F^M_{N_1N_2}=|\{M\supseteq U|\ U\cong N_2,\
M/U\cong N_1\}|$ are called Ringel-Hall numbers. Note that in this case the extension monoid product is $\{[N_1]\}*\{[N_2]\}=\{[M]|F^M_{N_1N_2}\neq 0\}$.

In \cite{szantoszoll} we have proved that extensions of preinjective (preprojective) Kronecker modules are field independent, so the extension monoid product of two preinjectives (preprojectives) can be described combinatorially. In this paper we generalize this result by showing that extensions of arbitrary Kronecker modules are field independent up to Segre classes, so all extension monoid products can be described combinatorially. In order to prove this result we formulate a variant of the well known Green formula for Ringel-Hall numbers, valid over an arbitrary (not only finite) field and we describe explicit formulas for some particular extension monoid products.

The last section surveys some combinatorial properties related with the embedding and extension of preinjective modules. We show that particular orderings known from partition combinatorics and their generalizations (such as dominance, weighted dominance, generalized majorization) play an important role in this context.

\section{\bf Green's formula rewritten}
Consider an acyclic quiver $Q$ and the path algebra $kQ$, where $k$ is an arbitrary field. Denote by $\Mod kQ$ the category of finite dimensional right modules over $kQ$. The aim of this section is to formulate a weaker version of Green's formula valid over an arbitrary field $k$ (not only over finite ones). The proof can be derived from \cite{rin2} or it follows directly from Lemma 2.9 and Lemma 2.11 in \cite{hub1} (using the fact that our algebra is hereditary).
\begin{theorem}{\rm(Green)} \label{GreenTheorem} For $M,N,X,Y\in\Mod kQ$

$\exists E\in\Mod kQ$ such that the cross $$\xymatrix{ &
Y\ar[d]^{\eta} & \\ N\ar[r]^{\nu} & E\ar[r]^{\mu}\ar[d]^{\xi} & M\\ & X & }$$
is exact (i.e. both the row and the column are short exact sequences) iff
$\exists A,B,C,D\in\Mod kQ$ such that the frame
$$\xymatrix{D\ar[r]^{\delta_Y}\ar[d]^{\delta_N} & Y\ar[r]^{\beta_Y} & B\ar[d]^{\beta_M}\\ N\ar[d]^{\gamma_N} & &
M\ar[d]^{\alpha_M}\\ C\ar[r]^{\gamma_X}& X\ar[r]^{\alpha_X} & A}$$
is exact (i.e. the edges of the frame are short exact sequences).

Moreover, in this case we also have the following $3\times 3$ commutative diagram with rows and columns short exact sequences and with the top left square a pull-back and the bottom right square a push-out. $$\xymatrix{D\ar[r]^{\delta_Y}\ar[d]^{\delta_N} & Y\ar[r]^{\beta_Y}\ar[d]^{\eta} & B\ar[d]^{\beta_M}\\ N\ar[r]^{\nu}\ar[d]^{\gamma_N} & E\ar[r]^{\mu}\ar[d]^{\xi} & M\ar[d]^{\alpha_M}\\ C\ar[r]^{\gamma_X}& X\ar[r]^{\alpha_X} & A}$$
\end{theorem}
There is an important corollary of Theorem \ref{GreenTheorem}:
\begin{corollary} \label{Greencor} For $M,N,X,Y\in\Mod kQ$ with
$\Ext^1(M,N)=0$ there is an exact sequence $0\to Y\to M\oplus N\to X\to 0$ iff $\exists A,B,C,D\in\Mod kQ$
such that the frame below is exact. $$\xymatrix{D\ar[r]\ar[d] &
Y\ar[r] & B\ar[d]\\ N\ar[d] & & M\ar[d]\\ C\ar[r]&
X\ar[r] & A}$$
\end{corollary}

\section{\bf Facts on Kronecker modules}

The indecomposables in $\Mod kK$ are divided into three families:
the preprojectives, the regulars and the preinjectives (see
\cite{assem},\cite{aus},\cite{rin1}).

The preprojective (respectively preinjective) indecomposable modules (up to isomorphism)
will be denoted by $P_n$ (respectively $I_n$), where $n\in\mathbb N$. The dimension vector of $P_n$ is $(n+1,n)$ and that of $I_n$ is $(n,n+1)$.

A module is preprojective (preinjective) if it is the direct sum of preprojective (preinjective) indecomposables.
For a partition $a=(a_1,\dots,a_n)$ we will use the notation $P_a:=P_{a_n}\oplus\dots\oplus P_{a_1}$ and $I_a:=I_{a_1}\oplus\dots\oplus I_{a_n}$.

The indecomposables which are neither preinjective nor preprojective are called regular. A module is regular if it is the direct sum of regular indecomposables.
The category of regular modules is an abelian, exact subcategory which decomposes into a direct sum of serial categories with Auslander-Reiten quiver of the form $ZA_{\infty}/1$, called homogeneous tube. These tubes are indexed by the closed points $x$ in the scheme $\mathbb{P}^1_k=\Proj k[X,Y]$. We denote by $\mathbb{H}_k$ the set of these points. A regular indecomposable of regular length $t$ lying on the tube $\mathcal{T}_x$ will be denoted by $R_x(t)$. Note that its unique regular composition series is $R_x(1)\subset\dots\subset R_x(t-1)\subset R_x(t)$. Also note that for the regular simple $R_x(1)$ its endomorphism ring $\End(R_x(1))$ is the residue field at the point $x$. The degree of this field over $k$ is called the degree of the point $x$ and denoted by $\deg x$. It follows that $\underline\dim R_x(t)=(t\deg x,t\deg x)$. In the case when $k$ is algebraically closed, the closed points of the scheme above all have degree 1 and can be identified with the points of the classical projective line over $k$.

For a partition $\lambda=(\lambda_1,\dots,\lambda_n)$ we define $R_x(\lambda)=R_x(\lambda_1)\oplus\dots\oplus R_x(\lambda_n)$ and denote by
$P$ (respectively $I$, $R$) a preprojective (respectively preinjective, regular) module.
We also define the set
$$\mathcal{R}_n=\{[R_{x_1}(a_1)\oplus...\oplus R_{x_r}(a_r)]\,|\,r,a_1,\dots,a_r\in\mathbb N^*,
x_1,...,x_r\in \mathbb{H}_k \text{ are pairwise different and }$$$$
a_1\deg {x_1}+...+a_r\deg {x_r}=n\}.$$
We will describe now (in our special context) the so-called decomposition symbol used by Hubery in \cite{hub2}.
A decomposition symbol $\alpha=(\mu,\sigma)$ consists of a pair of partitions denoted by $\mu$ (specifying a module without
homogeneous regular summands) and a multiset $\sigma=\{(\lambda^1, d_1),\dots,(\lambda^r, d_r)\}$, where $\lambda^i$ are partitions and $d_i\in\mathbb N^*$. The multiset $\sigma$ will be called a Segre symbol. Given a decomposition symbol $\alpha=(\mu,\sigma)$ (where $\mu=(c=(c_1,...,c_t),d=(d_1,...,d_s))$) and a field $k$, we define the decomposition class $S(\alpha,k)$ to be the set of isomorphism classes of modules of the form $M(\mu,k)\oplus R$, where $M(\mu,k)=P_{c_t}\oplus\dots\oplus P_{c_1}\oplus I_{d_1}\oplus\dots\oplus I_{d_s} $ is the $kK$-module (up to isomorphism uniquely) determined by $\mu$ and $R=R_{x_1}(\lambda^1)\oplus\dots\oplus R_{x_r}(\lambda^r)$
for some distinct points $x_1,\dots x_r\in\mathbb H_k$ such that $\deg x_i = d_i$. For a Segre symbol $\sigma$ let $S(\sigma,k):=S((\emptyset,\sigma),k)$. Trivially $S(\alpha,k)\cap S(\beta,k)=\emptyset$ for decomposition symbols $\alpha\neq\beta$. Also note that $\mathcal{R}_n=\cup S(\sigma,k)$, where the union is taken over all Segre symbols $\sigma=\{((a_1), d_1),\dots,((a_r), d_r)\}$ with $r\in\mathbb N^*$ and $a_1 d_1+...+a_r d_r=n$.

\begin{remark}\label{NrPointsHk}For $k$ finite with $q$ elements the number of points $x\in\mathbb H_k$ of degree 1 is $q+1$. The number of points $x\in\mathbb H_k$ of degree $l\geq 2$ is $N(q,l)=\frac1{d}\sum_{d|l}\mu(\frac{l}d)q^d$, where $\mu$ is the M\"obius function and $N(q,l)$ is the
number of monic, irreducible polynomials of degree $l$ over a field with $q$ elements. We can conclude that for a decomposition symbol $\alpha$ the polynomial  $n_{\alpha}(q)=|S(\alpha,k)|$ is strictly increasing in $q>1$ (see \cite{hub2}).
\end{remark}

The following well-known lemma summarizes some facts on Kronecker modules:
\begin{lemma}\label{KroneckerFactsLemma}

{\rm a)} Let $P$ be preprojective, $I$ preinjective and $R$ regular module. Then $\Hom ({R},{P})=\Hom ({I},{P})=\Hom ({I},{R})=
\Ext^1({P},{R})=\Ext^1({P},{I})= \Ext^1({R},{I})=0.$

{\rm b)} If $x\neq x'$, then
$\Hom(R_x(t),R_{x'}(t'))=\Ext^1(R_x(t),R_{x'}(t'))=0$ (i.e. the tubes are pairwise orthogonal).

{\rm c)} For $n\leq m$, we have $\dim _k\Hom (P_n,P_m)=m-n+1$ and
$\Ext^1(P_n,P_m)=0$; otherwise $\Hom (P_n,P_m)=0$ and $\dim
_k\Ext^1(P_n,P_m)=n-m-1$. In particular $\End(P_n)\cong k$ and
$\Ext^1(P_n,P_n)=0$.

{\rm d)} For $n\geq m$, we have $\dim _k\Hom (I_n,I_m)=n-m+1$ and
$\Ext^1(I_n,I_m)=0$; otherwise $\Hom (I_n,I_m)=0$ and $\dim
_k\Ext^1(I_n,I_m)=m-n-1$. In particular $\End(I_n)\cong k$ and
$\Ext^1(I_n,I_n)=0$.

{\rm e)} $\dim _k\Hom (P_n,I_m)=n+m$ and $\dim
_k\Ext^1(I_m,P_n)=m+n+2$.

{\rm f)} $\dim _k\Hom (P_n,R_x(t))=\dim _k\Hom (R_x(t),I_n)=d_pt$
and $\dim _k\Ext^1(R_x(t),P_n)=\dim _k\Ext^1(I_n,R_x(t))=t\deg x$.

{\rm g)} $\dim _k\Hom (R_x(t_1),R_x(t_2))=\dim
_k\Ext^1(R_x(t_1),R_x(t_2))=\min{(t_1,t_2)}\deg x$.

{\rm h)} For $P'$ a preprojective module every nonzero morphism $f:P_n\to P'$ is a monomorphism. If $R$ is regular
then for every nonzero morphism $f:P_n\to R$, $f$ is either a
monomorphism or $\Ima f$ is regular. In particular if $R$ is
regular simple and $\Ima f$ is regular then $f$ is an epimorphism.

{\rm i)} For $I'$ a preinjective module every nonzero morphism $f:I'\to I_n$ is an epimorphism. If $R$ is regular then for every nonzero morphism $f:R\to I_n$, $f$ is either an epimorphism or $\Ima f$ is regular. In particular if $R$ is
regular simple and $\Ima f$ is regular then $f$ is a monomorphism.
\end{lemma}

The defect of $M\in \Mod kK$ with dimension vector $(a,b)$ is
defined in the Kronecker case as $\partial M:=b-a$. Observe that if
$M$ is a preprojective (preinjective, respectively regular)
indecomposable, then $\partial M=-1$ ($\partial M=1$, respectively
$\partial M=0$). Moreover, for a short exact sequence $0\to M_1\to
M_2\to M_3\to 0$ in $\Mod kK$ we have $\partial M_2=\partial
M_1+\partial M_3$.

An immediate consequence of the facts above is the following:
\begin{corollary} \label{ProdPRI} For the partitions $c=(c_1,...,c_t)$ and $d=(d_1,...,d_s)$ we have that:  $$\{[P_c\oplus R_{x_1}(\lambda^1)\oplus\dots\oplus R_{x_r}(\lambda^r)\oplus I_d]\}=$$ $$\{[P_{c_t}]\}*\dots *\{[P_{c_1}]\}*\{[R_{x_1}(\lambda^1)]\}*\dots *\{[R_{x_r}(\lambda^r)]\}*\{[I_{d_1}]\}*\dots *\{[I_{d_s}]\}.$$
\end{corollary}

As stated in the beginning, we focus our attention on extensions of Kronecker modules, or equivalently on the products of
the form $\{[M]\}*\{[N]\}$. Using the corollary above we can see
that this iteratively reduces to the knowledge of the
following particular products:
$$\{[I_i]\}*\{[I_j]\},\ \{[P_i]\}*\{[P_j]\},\ \{[I_n]\}*\{[R_x(\lambda)]\},\ \{[R_x(\lambda)]\}*\{[P_n]\},\ \{[I_{n}]\}*\{[P_m]\},\ \{[R_x(\lambda)]\}*\{[R_x(\mu)]\}.$$

\section{\bf Particular extension monoid products and field independence in the general case}
In this section we will work in the category $\Mod kK$ with $k$ an arbitrary field.
We will analyze the field independence of extensions of arbitrary Kronecker modules. For this purpose we will describe the particular extension monoid products listed at the end of the previous section.

We start with the description of $\{[I_i]\}*\{[I_j]\}$ and $\{[P_i]\}*\{[P_j]\}.$
\begin{proposition} \label{ProdIiIj} We have:

$\{[I_i]\}*\{[I_j]\}=\left\{\begin{array}{cc}\{[I_i\oplus I_j]\}
&\text{ for } i-j\geq-1
\\ \{[I_j\oplus I_i],[I_{j-1}\oplus
I_{i+1}],...,[I_{j-[\frac{j-i}2]}\oplus I_{i+[\frac{j-i}2]}]\}
&\text{ for } i-j<-1\end{array}\right..$

Dually we have:

$\{[P_i]\}*\{[P_j]\}=\left\{\begin{array}{cc}\{[P_i\oplus P_j]\}
&\text{ for } i-j\leq-1
\\ \{[P_j\oplus P_i],[P_{j+1}\oplus
P_{i-1}],...,[P_{j+[\frac{i-j}2]}\oplus P_{i-[\frac{i-j}2]}]\}
&\text{ for } i-j>-1\end{array}\right..$
\end{proposition}
\begin{proof} For $k$ a finite field the formulas follow directly
from the corresponding formulas for the Ringel-Hall product (see
\cite{szanto1} for details). In \cite{szantoszoll} it is proven that
the possible middle terms in preinjective or preprojective short
exact sequences do not depend on the base field, so we are done.
\end{proof}

We describe now the product $\{[R_x(\lambda)]\}*\{[R_x(\mu)]\}$, where $\lambda$ and $\mu$ are partitions.
This is a classical result and it was studied in the equivalent context of
$p$-modules by T. Klein in \cite{klein}. So we have:
\begin{proposition}\label{RCommuteLambdaMu}$\{[R_x(\lambda)]\}*\{[R_x(\mu)]\}=\{[R_x(\mu)]\}*\{[R_x(\lambda)]\}=\{R_x(\nu)| c^{\nu}_{\lambda\mu}\neq
0\},$ where $c^{\nu}_{\lambda\mu}$ is the Littlewood-Richardson
coefficient (which is field independent).
\end{proposition}

Using our knowledge on Littlewood-Richardson coefficients we obtain
in particular the following:
\begin{corollary}\label{RCommuteLambdaN}$\{[R_x(\lambda)]\}*\{[R_x(n)]\}=\{[R_x(n)]\}*\{[R_x(\lambda)]\}=\{R_x(\nu)|
\nu-\lambda\text{ is a horizontal $n$-strip}\}.$
\end{corollary}

Using the field independence of the Littlewood-Richardson coefficients we also obtain:
\begin{corollary} \label{ProdR}For two Segre symbols $\sigma,\tau$ we have that $S(\sigma,k)*S(\tau,k)=\bigcup S(\rho,k)$, where the union is taken over a finite number of specific Segre symbols $\rho$, combinatorially (field independently) determined by the symbols $\sigma,\tau$.
\end{corollary}
\begin{proof} Using Proposition \ref{RCommuteLambdaMu} the field independent combinatorial nature of the product is clear. What remains to prove is that the product is the union of full Segre classes. Suppose that the class $[R_{x_1}(\lambda^1)\oplus\dots\oplus R_{x_r}(\lambda^r)]$ occurs in the product
for some distinct points $x_1,\dots x_r\in\mathbb H_k$ such that $\deg x_i = d_i$. We prove that in this case the whole Segre class $S(\{(\lambda^1, d_1),\dots,(\lambda^r, d_r)\},k)$ occurs in the product. For a component $R_{x_i}(\lambda^i)$ we have the following three possibilities:

a) it comes from the product $[R_{x_i}(\mu^i)][R_{x_i}(\nu^i)]$, where $(\mu^i,d_i)\in\sigma$, $(\nu^i,d_i)\in\tau$, so $c^{\lambda^i}_{\mu^i\nu^i}\neq 0$.

b) $(\lambda^i,d_i)\in\sigma$,

c) $(\lambda^i,d_i)\in\tau$.

Note that in any of the cases above, the component $R_{y_i}(\lambda^i)$ can be obtained in a similar way, where $y_i\in\mathbb H_k$ is arbitrary such that $\deg y_i = d_i$.
\end{proof}
\begin{remark} One can observe that the previous corollary is valid also in the case when one of the Segre classes are empty (due to the smallness of the field). See also Remark \ref{NrPointsHk}.
\end{remark}

Next we consider the products $\{[I_n]\}*\{[R_x(\lambda)]\}$ and $\{[R_x(\lambda)]\}*\{[P_n]\}.$ We need the following lemma:

\begin{lemma} \label{preproj} Let $P_n,P_m$ be preprojective indecomposables with $n<m$. Then
there is a short exact sequence $0\to P_n\to P_m\to X\to 0$ iff $X$
satisfies the following conditions:

{\rm i)} it is a regular module with $\underline\dim
X=\underline\dim P_m-\underline\dim P_n,$

{\rm ii)} if $R_x(t)$ and $R_{x'}(t')$ are two indecomposable
components of $X$ then $x\neq x'$.
\end{lemma}
\begin{proof} Suppose we have a short exact
sequence $0\to P_n\to P_m\to X\to 0$. We will check the conditions
i) and ii).

Condition i). Trivially, $\underline\dim X=\underline\dim
P_m-\underline\dim P_n$ and $\partial X=0$. Note that $X$ cannot
have preprojective components, since if $P''$ would be such an
indecomposable component, then $P_m\twoheadrightarrow P''\ncong P_m$
which is impossible due to Lemma \ref{KroneckerFactsLemma} h). So $X$ is regular.

Condition ii). Suppose $X=X'\oplus R_x(t_1)\oplus...\oplus
R_x(t_l).$ Then we have a monomorphism
$\Hom(X,R_x(1))\to\Hom(P_m,R_x(1)),$ so
$\dim_k\Hom(X,R_x(1))=\dim_k\Hom(X',R_x(1))+\sum_{i=1}^{l}\dim_k\Hom(R_x(t_i),R_x(1))\leq
\dim_k\Hom(P_m,R_x(1))=\deg x$ and $\dim_k\Hom(R_x(t_i),R_x(1))=\deg x$.
It follows that $l=1$.

Conversely suppose now that $X$ is a regular module satisfying
conditions i) and ii). It is enough to show that $P_m$ projects on
$X$, since for an epimorphism $f:P_m\to X$ we have that
$\partial\Ker f=-1$, so $\Ker f\cong P_n$. Notice first that there are no monomorphisms
$P_m\to X$ because
$\underline{\dim}X=\underline{\dim}P_m-\underline{\dim}P_n<\underline{\dim}P_m.$
For a nonzero $f:P_m\to X$ we have the short exact sequence
$0\to\Ker f\to P_m\to\Ima f\to 0$. Since $\Ker f\subseteq P_m$  we
have that $\Ker f$ is preprojective (so with negative defect) and is
not 0 (because $f$ is not mono) and  $\Ima f\subseteq X$ implies
that $\Ima f$ may contain preprojectives and regulars as direct
summands (and it is nonzero since $f$ is nonzero). The equality
$\partial\Ker f+\partial\Ima f=\partial P_m=-1$ gives us
$\partial\Ima f=0$, so $\Ima f$ is regular.

For $X=R_x(t)$ we have that $\Hom(P_m,X)\neq 0$ (see Lemma \ref{KroneckerFactsLemma} f)). If there are no
epimorphisms in $\Hom(P_m,R_x(t))$ then using the remarks above and
the uniseriality of regulars we would have
$\Hom(P_m,R_x(t))\cong\Hom(P_m,R_x(t-1))$ a contradiction. So we
have an epimorphism $P_m\to X$.

Suppose now that $X=R_{x_1}(t_1)\oplus...\oplus R_{x_l}(t_l)$. From
the discussion above we have the epimorphisms $f_i:P_m\to
R_{x_i}(t_i)$. Let $f:P_m\to X$, $f(x)=\sum f_i(x)$ the diagonal
map. We have that $\Ima f$ is regular so due to uniseriality $\Ima
f=R_{x_1}(t'_1)\oplus...\oplus R_{x_l}(t'_l)$ with
$R_{x_i}(t'_i)\subseteq R_{x_i}(t_i)$. Since $f_i=p_if$ are
epimorphisms we have that $R_{x_i}(t'_i)= R_{x_i}(t_i)$ so $f$ is an
epimorphism.
\end{proof}
\begin{proposition} \label{ProdInRlambda}We have:
$$\{[R_x(\lambda)]\}*\{[P_n]\}=\{[P_{n+t\deg x}\oplus R_x(\mu)]\,|\,\text{where $\lambda-\mu$ is a horizontal strip of length $t$, for some
$t\in\mathbb N$}\}.$$
Dually we have:
$$\{[I_n]\}*\{[R_x(\lambda)]\}=\{[R_x(\mu)\oplus I_{n+t\deg x}]\,|\,\text{where $\lambda-\mu$  is a horizontal strip of length $t$, for some
$t\in\mathbb N$}\}.$$
\end{proposition}
\begin{proof}We prove the first formula.
Suppose we have a short exact sequence $$0\to P_n\to X\overset{g}\to
R_x(\lambda)\to 0.$$ Note that we can't have preinjective components
in $X$ (since due to Lemma \ref{KroneckerFactsLemma} a) they would embed into $\Ker g\cong P_n$). Since
$\partial X=-1$, it follows using Lemma \ref{KroneckerFactsLemma} c) that $X$ is of
the form $X=P_{n+t\deg x}\oplus R_x(\mu)$ where $\mu$ is a partition
with $|\mu|\leq |\lambda|$ and $t=|\lambda|-|\mu|$.

If $\mu=(0)$ then by Lemma \ref{preproj} we have an exact sequence $0\to
P_n\to P_{n+t\deg x}\to R_x(\lambda)\to 0$ iff $\lambda=(t)$ i.e. iff
$\lambda-(0)$ is a horizontal $t$-strip.

If $\mu\neq (0)$ then we apply Corollary \ref{Greencor} with choices
$X=R_x(\lambda)$, $Y=P_n$, $M=P_{n+t\deg x}$ and
$N=R_x(\mu)$. It follows that we have an exact sequence $0\to
P_n\to P_{n+t\deg x}\oplus R_x(\mu)\to R_x(\lambda)\to 0$ iff $\exists
A,B,C,D\in\Mod kK$ such that the frame below is exact.
$$\xymatrix{D\ar[r]\ar[d] & P_n\ar[r] & B\ar[d]\\ R_x(\mu)\ar[d] & &
P_{n+t\deg x}\ar[d]\\ C\ar[r]& R_x(\lambda)\ar[r] & A}$$ By Lemma \ref{KroneckerFactsLemma} a)
$B,D$ are preprojectives or 0. Note that $B,D$ can't be both
preprojectives (due to the defect)  and also if $B=0$ then $A=P_{n+t\deg x}$, a
contradiction since $R_x(\lambda)$ would project on a
preprojective. This means that we must have $D=0$, so $B=P_n$,
$C=R_x(\mu)$ and using Lemma \ref{preproj} it follows that $A=R_x(t)$ (where
$t=|\lambda-\mu|$). So we have an exact sequence $0\to P_n\to
P_{n+t\deg x}\oplus R_x(\mu)\to R_x(\lambda)\to 0$ iff the frame below
is exact. $$\xymatrix{0\ar[r]\ar[d] & P_n\ar[r] & P_n\ar[d]\\
R_x(\mu)\ar[d] & & P_{n+t\deg x}\ar[d]\\ R_x(\mu)\ar[r]&
R_x(\lambda)\ar[r] & R_x(t)}$$ Applying Corollary \ref{RCommuteLambdaN} it follows
that $\lambda-\mu$ must be a horizontal $t$-strip.
\end{proof}
For $\lambda=(m)$ we have in particular:
\begin{corollary} $\{[R_x(m)]\}*\{[P_n]\}=\{[P_{n+i\deg x}\oplus R_x(m-i)]|\text{ where
$i=\overline{0,m}$}\}$.

Dually

$\{[I_n]\}*\{[R_x(m)]\}=\{[R_x(m-i)\oplus I_{n+i\deg x}]|\text{ where
$i=\overline{0,m}$ }\}$.

\end{corollary}

Applying the previous corollary inductively, we obtain the following:
\begin{corollary}\label{ProdIR} a) We have $\mathcal{R}_n*\{[P_m]\}=(\{[P_m]\}*\mathcal{R}_n)\cup(\{[P_{m+1}]\}*\mathcal{R}_{n-1})\cup\dots\cup\{[P_{m+n}]\}$.

b) For a Segre symbol $\sigma=\{(\lambda^1, d_1),\dots,(\lambda^r, d_r)\}$ we have that $S(\sigma,k)*\{[P_m]\}=\bigcup \{[P_{m+t}]\}*S(\tau,k)$, where the union is taken over all Segre symbols of the form $\tau=\{(\mu^1, d_1),\dots,(\mu^r, d_r)\}$ with $\lambda^i-\mu^i$ a horizontal strip of lenght $t_i$ and $t=t_1d_1+\dots+t_rd_r$.

The preinjective version of the formulas above follows dually.
\end{corollary}

Finally we consider the product $\{[I_n]\}*\{[P_m]\}$.
\begin{proposition} \label{ProdIP} We have

$\{[I_n]\}*\{[P_m]\}=\mathcal{R}_{n+m+1}\cup\{[P_m\oplus I_n]\}.$
\end{proposition}
\begin{proof} Suppose first that
$X\ncong P_m\oplus I_n$. Then we prove that there is an exact sequence of the form $0\to
P_m\to X\to I_n\to 0$ iff $X$ is a regular module having indecomposable components from pairwise different tubes and $\underline\dim X=\underline\dim P_m+\underline\dim I_n$.

Suppose we have a short exact sequence $0\to P_m\overset{f}\to
X\overset{g}\to I_n\to 0$. Then $\underline\dim X=\underline\dim
P_m+\underline\dim I_n$ and $\partial X=\partial P_m+\partial
I_n=0$. Suppose $X=P'\oplus R\oplus I'$ (where $P'$, $R$ and $I'$
are preprojective, preinjective and regular modules). Note that
$p_{P'}f:P_m\to P'$ must be nonzero so it is a monomorphism (see
Lemma \ref{KroneckerFactsLemma} h)) which means that $\underline\dim P_m\leq\underline\dim
P'$. In the same way $fq_{I'}:I'\to I_n$ must be nonzero so it is an
epimorphism (see Lemma \ref{KroneckerFactsLemma} i)) which means that $\underline\dim
I_n\leq\underline\dim I'$. But $\underline\dim P_m+\underline\dim
I_n=\underline\dim P'+\underline\dim R+\underline\dim I'$ which
implies $R=0$ and $p_{P'}f$, $fq_{I'}$ are isomorphisms, so $X\cong
P_m\oplus I_n$ a contradiction. This means that $X$ is regular.

Suppose $X=X'\oplus R_x(t_1)\oplus...\oplus R_x(t_l)$, then we have
the monomorphism $0\to\Hom(X,R_x(1))\to\Hom(P_m,R_x(1)),$ since
$\Hom(I_n,R_x(1))=0$. It follows that
$$\dim_k\Hom(X,R_x(1))=\dim_k\Hom(X',R_x(1))+\sum_{i=1}^{l}\dim_k\Hom(R_x(t_i),R_x(1))\leq\dim_k\Hom(P_m,R_x(1))=\deg x$$ and $\dim_k\Hom(R_x(t_i),R_x(1))=\deg x$, so $l=1$.

Conversely, suppose that $X$ is a regular module having indecomposable components from pairwise different tubes and $\underline\dim
X=\underline\dim P_m+\underline\dim I_n$. Repeating the proof of Lemma 3.2. in
\cite{szanto2} the existence of an exact sequence $0\to
P_m\to X\to I_n\to 0$ follows.

\end{proof}

Using the previous results on particular extension monoid products (more precisely Corollaries \ref{ProdPRI}, \ref{ProdR}, \ref{ProdIR} and Proposition \ref{ProdIP}) it follows inductively that the extension monoid product of Kronecker modules is field independent in general up to Segre classes. More precisely we obtain the following theorem:
\begin{theorem} For two decomposition symbols $\alpha,\beta$ we have that $S(\alpha,k)*S(\beta,k)=\bigcup S(\gamma,k)$, where the union (which is disjoint) is taken over a finite number of specific decomposition symbols $\gamma$ combinatorially (field independently) determined by the symbols $\alpha,\beta$.
\end{theorem}
\begin{remark} One can observe that the theorem above is valid also in the case when one of the decomposition classes is empty (due to the smallness of the field). See also Remark \ref{NrPointsHk}. Also
\end{remark}

\section{\bf Combinatorial aspects of extensions of preinjective Kronecker modules}
There are some very interesting combinatorial properties related with the embedding and extension of preinjective modules. Particular orderings known from partition combinatorics and their generalizations (such as dominance, weighted dominance, generalized majorization) play an important role in this context.
We recall first the definition of these orderings.
\begin{definition} Let $a=(a_1,...a_n),b=(b_1,...,b_n)\in\mathbb Z^n$ .

\noindent The dominance partial ordering is defined as follows (see \cite{Macd}):
$$a\leqslant b\text{ iff }a_1\leq b_1,\  a_1+a_2\leq b_1+b_2,\ \dots,\ a_1+...+a_{n-1}\leq
b_1+...+b_{n-1}\text{ and }a_1+....+a_n\leq b_1+...+b_n.$$
In case $a_1+....+a_n=b_1+...+b_n$ (for example when $a,b$ are partitions of the same number) we will use the notation $a\preccurlyeq b$.

\noindent The weighted dominance partial ordering is defined as follows (see \cite{szantoszoll}):
$$a\ll b\text{ iff } (a_1,2a_2,...,na_n)\leqslant (b_1,2b_2,...,nb_n).$$

\noindent Following Baraga\~na, Zaballa, Mondi\'e, Dodig,  Sto\u si\'c one can define the so-called generalized majorization (see \cite{dodig1},\cite{dodig2}). This generalization of the dominance ordering of partitions plays an important role in the combinatorial background of matrix pencil completion problems. Consider the partitions $a=(a_1,...,a_n),b=(b_1,...,b_m),c=(c_1,...,c_{m+n})$. Then we say that the pair $(b,a)$ is a generalized majorization of $c$ (and denote it by $c\prec (b,a)$) iff $$b_i\geq c_{i+n},\ i=\overline{1,m},$$ $$ \sum_{i=1}^{m}b_{i}+\sum_{i=1}^{n}a_{i}=\sum_{i=1}^{m+n}c_{i},$$$$\sum_{i=1}^{h_{q}}c_{i}-\sum_{i=1}^{h_{q}-q}b_{i}\leq\sum_{i=1}^{q}a_{i},\ q=\overline{1,n},$$$$where\ h_{q}:=\min\{i|b_{i-q+1}<c_{i}\},\ q=\overline{1,n}.$$ Adopting the convention that $c_i,b_i=+\infty$ for $i\leq 0$, $c_i=-\infty$, for $i>m+n$ $b_i=-\infty$, for $i>m$ and $\sum_{i=a}^b=0$ in case $a>b$ one can see that the indices $h_q$ and the sums above are all well defined. Moreover, we have that $q\leq h_q\leq q+m$ and $h_1<h_2<\dots<h_n$.
The term generalized majorization is motivated by the fact, that for $m=0$ the generalized majorization reduces to the dominance ordering $c\preccurlyeq a$.
\end{definition}

Consider the partitions $a=(a_1,...,a_n)$, $b=(b_1,...,b_m)$ and $c=(c_1,...,c_{m+n})$. Using the definition of the generalized majorization we define inductively $x_1:=\sum_{i=1}^{h_{1}}c_{i}-\sum_{i=1}^{h_{1}-1}b_{i},\ x_q:=\sum_{i=1}^{h_{q}}c_{i}-\sum_{i=1}^{h_{q}-q}b_{i}-x_1-\dots-x_{q-1}$ for $q=\overline{1,n}$. If $x:=(x_1,\dots,x_n)\in\mathbb Z^n$ then observe that $x$ depends only on $b,c$ and not on $a$. The following proposition gives an another connection between the notion of generalized majorization and the dominance ordering.
\begin{proposition} \label{GenMajChar} Suppose that for the partitions $a,b,c$ above we have that $b_i\geq c_{i+n}$ for $i=\overline{1,m}$ and $\sum_{i=1}^{m}b_{i}+\sum_{i=1}^{n}a_{i}=\sum_{i=1}^{m+n}c_{i}$. Then $c\prec (b,a)$ iff $x\preccurlyeq a$.
\end{proposition}
\begin{proof} If $c\prec (b,a)$ then it follows from the definition that $x\leqslant a$. Since $\sum_{i=1}^{m}b_{i}+\sum_{i=1}^{n}a_{i}=\sum_{i=1}^{m+n}c_{i}$, we obtain that $\sum_{i=1}^{n}x_{i}+\sum_{i=1}^{m+n-h_n}(c_{h_n+i}-b_{h_n-n+i})=\sum_{i=1}^{n}a_{i}$. But we have that $c_{h_n+i}-b_{h_n-n+i}\leq 0$ and $\sum_{i=1}^{n}x_{i}\leq \sum_{i=1}^{n}a_{i}$, so $c_{h_n+i}=b_{h_n-n+i}$ for $i=\overline{1,m+n-h_n}$ and $\sum_{i=1}^{n}x_{i}=\sum_{i=1}^{n}a_{i}$, which means that $x\preccurlyeq a$. The converse statement is trivial.
\end{proof}

It follows from the proof above that for $n=1$ the condition $c\prec (b,a)$ is equivalent to $$b_i\geq c_{i+1},\ i=\overline{1,m},$$ $$ \sum_{i=1}^{m}b_{i}+a_1=\sum_{i=1}^{m+1}c_{i},$$$$b_i=c_{i+1},\ i\geq h_1,$$$$ where\ h_{1}:=\min\{i|b_{i}<c_{i}\}.$$ In this case we speak about an elementary generalized majorization and denote it by $c\prec_1 (b,a)$.

A result by Dodig and Sto\u si\'c in \cite{dodig1} shows that one can decompose the generalized majorization into a ``composition" of elementary generalized majorizations. More precisely we have:
\begin{proposition}{\rm(\cite{dodig1})} \label{DodigElementaryStep} We have $c\prec (b,a)$ iff there is a sequence of partitions $d^j=(d^j_1,\dots,d^j_{m+j})$, $j=\overline{1,n}$ with $d^0=b$ and $d^n=c$ such that $d^j\prec_1(d^{j-1},a_j)$ for $j=\overline{1,n}$.
\end{proposition}

We mention next two existing results on extensions of preinjectives which involve the combinatorial notions above.
\begin{proposition}{\rm(Sz\'ant\'o \cite{szanto})}\label{ProdIincreasing} Suppose $a=(a_1,\dots,a_n)$ is a partition. Then
$$\{[I_{a_n}]\}*\dots*\{[I_{a_1}]\}=\{[I_{\alpha}]| \alpha\preccurlyeq a\}.$$
\end{proposition}
Remember that $\{[I_{a_1}]\}*\dots*\{[I_{a_n}]\}=\{[I_a]\}$ and denote the reversed product $\{[I_{a_n}]\}*\dots*\{[I_{a_1}]\}=\{[I_\alpha]| \alpha\preccurlyeq a\}$ by $\mathcal{I}_a$.
\begin{proposition}{\rm(Sz\'ant\'o, Sz\"oll\H osi \cite{szantoszoll})} Consider the preinjective modules $I'=(a_nI_n)\oplus...\oplus(a_0I_0)$,
$I=(b_nI_n)\oplus...\oplus(b_0I_0)$ (where $a_i,b_j\in\mathbb N$ and $a^2_n+b^2_n\neq 0$).
Let $a=(a_1,\dots,a_n)$ and $b=(b_1,\dots b_n)$.

Then there is a monomorphism $I'\to I$ iff $a_0\leq b_0$ and $a\ll b$.
\end{proposition}

The following result connects the notion of generalized majorization to extensions of preinjectives.
\begin{proposition}Consider the partitions $a=(a_1,...,a_n),b=(b_1,...,b_m),c=(c_1,...,c_{m+n})$. Then $$c\prec (b,a)\text{ iff } [I_c]\in\mathcal{I}_a*\{[I_b]\}$$
\end{proposition}
\begin{proof} The equality $\mathcal{I}_a*\{[I_b]\}=\{[I_\alpha]| \alpha\preccurlyeq a\}*\{[I_b]\}$ is clear using Proposition \ref{ProdIincreasing}. The case $n=1$ easily follows using the definition of the elementary generalized majorization and Proposition \ref{ProdIiIj}. The general case is then a consequence of Proposition \ref{DodigElementaryStep}.
\end{proof}
\begin{remark} The result above is also proved in \cite{szoll} by Sz\"oll\H osi. We obtain in this way an independent proof also for Proposition \ref{DodigElementaryStep} by Dodig and Sto\u si\'c.
\end{remark}
Using all the results above we finally give a new characterization of the embedding of preinjective Kronecker modules.

Consider the partitions $b=(b_1,...,b_m)$, $c=(c_1,...,c_{m+n})$ and as before let $x=(x_1,\dots,x_n)\in\mathbb Z^n$ with $x_1=\sum_{i=1}^{h_{1}}c_{i}-\sum_{i=1}^{h_{1}-1}b_{i},\ x_q=\sum_{i=1}^{h_{q}}c_{i}-\sum_{i=1}^{h_{q}-q}b_{i}-x_1-\dots-x_{q-1}$ for $q=\overline{1,n}$.
\begin{proposition}There is a monomorphism $I_b\to I_c$ iff $b_i\geq c_{i+n}$, for $i=\overline{1,h_n-n}$, $b_i=c_{i+n}$, for $i=\overline{h_n-n+1,m}$ and there is a partition $a$ such that $x\preccurlyeq a$. Moreover if $a$ is minimal (using the dominance ordering), then $I_a$ is a factor of the embedding $I_b\to I_c$, i.e. $0\to I_b\to I_c\to I_a\to 0$ is exact.
\end{proposition}
\begin{proof} If there is a monomorphism $I_b\to I_c$, then we have an exact sequence $0\to I_b\to I_c\to I_a\to 0$ with $a=(a_1,\dots,a_n)$ a partition. But then $[I_c]\in\{[I_{a}]\}*\{[I_b]\}\subseteq\mathcal{I}_a*\{[I_b]\}$, so $c\prec (b,a)$ and the assertion follows using Proposition \ref{GenMajChar}. Conversely, the given conditions guarantee by Proposition \ref{GenMajChar} that $c\prec (b,a)$, so $[I_c]\in\mathcal{I}_a*\{[I_b]\}$, which means that we have an exact sequence $0\to I_b\to I_c\to I_{a'}\to 0$, for an $a'\preccurlyeq a$.

In case $a$ is minimal, then as above $0\to I_b\to I_c\to I_{a'}\to 0$, for an $a'\preccurlyeq a$. Since $x\preccurlyeq a'$ and $a$ is minimal, it follows that $a'=a$.
\end{proof}

{\it Acknowledgements.} This work was supported by the Bolyai Scholarship of the Hungarian Academy of Sciences and Grant PN-II-ID-PCE-2012-4-0100.

\bigskip
\emph{Csaba Sz\'ant\'o}

"Babe\c s-Bolyai" University Cluj-Napoca

Faculty of Mathematics and Computer Science

Str. Mihail Kogalniceanu nr. 1

R0-400084 Cluj-Napoca

Romania

e-mail: szanto.cs@gmail.com


\begin{thebibliography}{99}
\bibitem{assem} {\it I. Assem, D. Simson, A. Skowronski}, Elements of
Representation Theory of Associative Algebras, Volume 1: Techniques
of Representation Theory. LMS Student Texts (No. 65) (Cambridge
Univ. Press 2006).
\bibitem{aus} {\it M. Auslander, I. Reiten, S. Smalo},
Representation Theory of Artin Algebras, {\it Cambridge Stud. in
Adv. Math.} 36 (Cambridge Univ. Press 1995).
\bibitem{dodig1} {\it M. Dodig, M. Sto\u si\' c}, On convexity of polynomial
paths and generalized majorizations, Electron. J. Combin., 17(1)(2010), R61.
\bibitem{dodig2} {\it M. Dodig, M. Sto\u si\'c}, On properties of the
generalized majorization, Electron. J. Linear Algebra, 26(2013), 471--509.
\bibitem{hub1} {\it A. Hubery}, Ringel-Hall algebras, lecture notes.
\bibitem{hub2} {\it A. Hubery}, Hall polynomials for affine quivers, Represent. Theory 14 (2010).
\bibitem{klein} {\it T. Klein}, The multiplication of Schur-functions and extensions of
$p$-modules, J. London Math. Soc. 43 (1968), 280-284.
\bibitem{Macd} {\it I. G. Macdonald},
Symmetric Functions and Hall Polynomials, Clarendon Press Oxford
1995.
\bibitem{reineke} {\it M. Reineke}, The monoid of families of quiver
representations. Proc. LMS 84 (2002), 663-685.
\bibitem{rin1} {\it C. M. Ringel},
Tame algebras and Integral Quadratic Forms, {\it Lect. Notes Math.}
1099 (Springer 1984).
\bibitem{rin2} {\it C. M. Ringel},
Green's Theorem on Hall Algebras, preprint.
\bibitem{szanto} {\it Cs. Sz\'ant\'o}, On the Hall product of preinjective Kronecker modules, Mathematica (Cluj),48(71), No.2 (2006),pp. 203--206.
\bibitem{szanto1} {\it Cs. Sz\'ant\'o}, Hall numbers and the
composition algebra of the Kronecker algebra. Algebras and
Representation Theory 9 (2006), 465-495.
\bibitem{szanto2} {\it Cs. Sz\'ant\'o}, On some Ringel-Hall products in tame
cases, Journal of Pure and Applied 	Algebra, 216 (2012), pp 2069-2078.
\bibitem{szantoszoll} {\it Cs. Sz\'ant\'o, I. Sz\"oll\H osi}, On preprojective short exact
sequences in the Kronecker case, Journal of Pure and Applied Algebra, 216 (2012), pp 1171-1177.
\bibitem{szoll} {\it I. Sz\"oll\H osi}, On the combinatorics of extensions of preinjective Kronecker modules, to  appear.
\end{thebibliography}
\end{document}